\documentclass[11pt]{amsart}

\usepackage{amssymb,mathtools}
\usepackage{enumitem}
\usepackage[hidelinks]{hyperref}
\numberwithin{equation}{section}
\hypersetup{pdftitle={Rank Stabilization for Sparse Coordinate Completion of Unit-Norm Tight Frames},pdfauthor={Dongwei Li}}
\newtheorem{theorem}{Theorem}[section]
\newtheorem{lemma}[theorem]{Lemma}
\newtheorem{proposition}[theorem]{Proposition}
\newtheorem{corollary}[theorem]{Corollary}
\theoremstyle{definition}
\newtheorem{example}[theorem]{Example}
\theoremstyle{remark}
\newtheorem{remark}[theorem]{Remark}

\newcommand{\Sym}{\operatorname{Sym}}
\newcommand{\Symo}{\operatorname{Sym}_0}
\newcommand{\rank}{\operatorname{rank}}
\newcommand{\tr}{\operatorname{tr}}
\newcommand{\im}{\operatorname{im}}
\newcommand{\Span}{\operatorname{span}}
\newcommand{\C}{\mathbb C}
\newcommand{\OO}{\mathcal O}
\newcommand{\DD}{\mathcal D}
\newcommand{\RR}{\mathcal R}

\newcommand{\off}{\mathrm{off}}
\newcommand{\Nd}{N_d}

\title[Rank Stabilization for Sparse Coordinate Completion]{Rank Stabilization for Sparse Coordinate Completion of Unit-Norm Tight Frames}
\author{Dongwei Li}
\address{School of Mathematics, Hefei University of Technology, Hefei 230601, China}
\email{dongweili@hfut.edu.cn}
\subjclass[2020]{Primary 42C15, 05B35; Secondary 15A83, 14M99}
\keywords{sparse coordinate completion, finite unit-norm tight frames, algebraic matroids, rank stabilization, generic fiber dimension, matroid union}

\begin{document}

\begin{abstract}
We determine generic completion-fiber dimensions for unit-norm tight frames
when each partially observed column has exactly two missing coordinates.
We work on the complex algebraic variety defined by the real frame equations.
The missing coordinate pairs form a labelled multigraph.  For every $d\ge4$
and every frame length $R\ge N_d:=\binom{d+1}{2}-1$, we prove that the
missing-pair Jacobian matroid is the Rado matroid of an explicit subspace
arrangement.  Its rank is therefore given by a graph-theoretic minimum
formula and is independent of $R$; the complementary rank defect is the
generic fiber dimension.  The main step is a realization theorem at length
$N_d$.  It combines a transversal--graphic matroid partition with a
zero-coordinate moment submersion and a compatible decomposition of a
generic residual matrix.  Basis extension and stability under fully observed
columns then yield the rank formula at every larger length.  We also give
an explicit defect formula in dimension four and a local real counterpart
on the smooth real frame locus.  A common-rotation obstruction bounds the
smallest uniform stabilization threshold from below by $2d-1$.
\end{abstract}

\maketitle

\section{Introduction}\label{sec:intro}

Finite unit-norm tight frames (FUNTFs) provide redundant representations
with useful reconstruction and erasure properties
\cite{BenedettoFickus,CasazzaKovacevic,BodmannPaulsen,Li2026}.
Their geometry and construction have been studied through differential,
algebraic, and symplectic methods
\cite{DykemaStrawn,CahillMixonStrawn,CahillFickusMixonPoteetStrawn,
NeedhamShonkwiler}; see also \cite{Waldron} for general background.
Here we ask how the positions of missing scalar coordinates determine the
dimension of the family of possible frame completions.  This question
belongs to the algebraic approach to matrix completion
\cite{SingerCucuringu,KiralyTheranTomioka,Tsakiris2024}, in which coordinate
projections and their algebraic matroids encode generic identifiability.

For integers $R>d\ge2$, let
\[
X_{d,R}
:=
\left\{
W=[w_1\ \cdots\ w_R]\in\C^{d\times R}
\;\middle|\;
\begin{array}{l}
WW^T=\dfrac{R}{d}I_d,\\[2pt]
w_a^Tw_a=1\quad(a=1,\ldots,R)
\end{array}
\right\}.
\]
Here $I_d$ is the identity matrix.  Transpose is used throughout: ``unit
norm'' means the bilinear equation $w^Tw=1$, rather than a Hermitian norm
condition.  This is the complex Zariski closure of the real FUNTF locus
\cite[Section~1 and equation~(1.1)]{BFR2020}.  Our main statements concern
complex algebraic completion.  Corollary~\ref{cor:real} gives the corresponding
rank and local fiber dimension on a dense open subset of the smooth real
locus; global real completion counts are a separate question.

Farnsworth and Rodriguez \cite{FarnsworthRodriguez} studied homogenized
FUNTF varieties and degrees of coordinate projections.  Bernstein,
Farnsworth, and Rodriguez \cite{BFR2020} developed the algebraic-matroid
formulation for $X_{d,R}$, characterized its bases in dimension three, and
obtained structural restrictions in higher dimensions.  We study the family
of patterns in which each partially observed column has two missing entries.
For every $d\ge4$ we determine its entire missing-pair Jacobian rank function
once $R\ge N_d$.  In particular, the result measures positive-dimensional
completion ambiguity as well as detecting finite fibers.

The new geometric input is the realization of every admissible critical
pattern by an actual tight frame.  Rado's theorem computes the rank for
independently chosen vectors in certain local subspaces, but frame columns
are coupled by a fixed matrix moment.  We prove that this coupling creates
no additional generic dependence at the stated frame lengths.  Combining
the critical realization with matroid basis extension and the column-extension
observation of \cite[Remark~4.6]{BFR2020} gives the full rank formula.

Write $\Sym(d)$ for the space of complex symmetric $d\times d$
matrices and
\[
\Symo(d):=\{A\in\Sym(d):\tr A=0\}.
\]
Put
\[
\Nd=\binom{d+1}{2}-1=\dim\Symo(d),\qquad d\ge4.
\]
The number $\Nd$ is intrinsic to the trace-free tightness equations.  By
\cite[Theorem~3.3]{BFR2020}, $X_{d,R}$ is irreducible for $R\ge d+2>4$;
in particular this holds throughout the regime $R\ge\Nd$ considered here.
Throughout the paper, a \emph{generic} assertion means that it holds on a
nonempty Zariski-open subset of the relevant irreducible variety.  When
several generic conditions are used simultaneously, we take their finite
intersection.

Our main result concerns every frame length $R\ge\Nd$.  Choose an arbitrary set
$S\subseteq[R]$ of $m$ frame columns.  In each selected column exactly two
coordinates are unobserved, while every coordinate in the other columns is
observed.  If the missing rows of column $a\in S$ are $i$ and $j$, record
that column as an edge $e_a=\{i,j\}$.  Thus the columnwise sparse missing pattern is an
edge-labelled $m$-edge multigraph on $[d]:=\{1,\dots,d\}$.  This is the smallest nontrivial columnwise missing pattern for the
infinitesimal completion problem.  With only one missing coordinate, the
linearized unit-norm equation generically forces that missing tangent
coordinate to vanish.  With two missing coordinates, the same equation leaves
a one-dimensional internal motion generated by a coordinate-plane rotation.
These local rotations are nevertheless coupled by the global tight-frame
moment equation, so their independence is not a purely local question.  The
lower-redundancy example below shows that this coupling can create additional
global dependencies; the stabilization theorem is therefore not automatic
from the local linearization.

For an edge submultiset $T\subseteq S$, let $v(T)$ be the number of incident
vertices and let $c(T)$ be the number of nontrivial connected components of
its simple support.  Set
\[
\rho_d(T)=\frac{v(T)(2d-v(T)+1)}2-c(T),
\qquad
\rho_d(\varnothing)=0.
\]
Let $E_{ij}$ denote the standard matrix unit.  For distinct $i,j$ set
\[
J_{ij}:=E_{ij}-E_{ji},
\]
so that $J_{ji}=-J_{ij}$.  For $1\le i<j\le d$ set
\[
L_{ij}:=\{[J_{ij},X]:X\in\Sym(d)\}\subseteq\Symo(d),
\]
where $[A,B]:=AB-BA$ denotes the commutator.  For an unordered edge
$e=\{i,j\}$ with $i<j$, write $J_e:=J_{ij}$ and $L_e:=L_{ij}$.  Thus the
restricted Jacobian vector associated with a selected column $a$ is
\[
q_a(W):=[J_{e_a},w_aw_a^T]\in\Symo(d).
\]
Let $\pi_S$ be the coordinate projection that forgets precisely the two
coordinates indexed by $e_a$ in every selected column $a\in S$, and put
$Y_S=\overline{\pi_S(X_{d,R})}$.  A generic completion fiber means a fiber
over a point in a suitable nonempty Zariski-open subset of $Y_S$; observations
are thus assumed to be compatible with the frame equations.
For a fixed
labelled set $S$, we call the generic linear matroid represented by
$\{q_a(W):a\in S\}$ the \emph{missing-pair Jacobian matroid}.  Its ground
elements are missing coordinate pairs, so it should be distinguished from the
ordinary algebraic matroid of $X_{d,R}$, whose ground elements are individual
scalar coordinates.

The main theorem identifies this Jacobian matroid with the Rado matroid
of the subspaces $L_e$ whenever $R\ge\Nd$.

\begin{theorem}[High-redundancy rank stabilization and fiber dimension]\label{thm:main}
Let $d\ge4$, $R\ge\Nd$, and let $S\subseteq[R]$ be any set of frame
columns with prescribed missing-pair labels $e_a$.  Then, for generic
$W\in X_{d,R}$,
\[
\rank\{q_a(W):a\in S\}
=
r_d(S):=
\min_{T\subseteq S}
\bigl(|S\setminus T|+\rho_d(T)\bigr).
\]
Moreover the generic fiber of $\pi_S:X_{d,R}\to Y_S$ has dimension
\[
\delta_d(S)
=
\max_{T\subseteq S}
\bigl(|T|-\rho_d(T)\bigr).
\]
For every $S'\subseteq S$, define $r_d(S')$ by the same formula with
$S'$ in place of $S$.  There is a single nonempty Zariski-open subset of
$X_{d,R}$, depending on the fixed labelled set $S$, on which
$\rank\{q_a(W):a\in S'\}=r_d(S')$ simultaneously for every
$S'\subseteq S$.
\end{theorem}

Because the formula holds simultaneously for every subcollection of $S$,
Theorem~\ref{thm:main} determines the entire missing-pair Jacobian matroid.
Its relation to the ordinary algebraic matroid on scalar coordinates is
explicit:
\[
\operatorname{rank}_{\rm alg}(\text{observed coordinates})
=\dim \overline{\pi_S(X_{d,R})}
=\dim X_{d,R}-|S|+r_d(S).
\]
By the fiber-dimension theorem, the middle term is the dimension of the image
closure; equivalently it is the algebraic-matroid rank of the observed
coordinate set.  Thus the same combinatorial function controls both the
grouped missing-pair rank and the completion geometry.  In particular,
$r_d(S)$ and $\delta_d(S)$ are independent of the total frame length $R$ for
$R\ge\Nd$.

For the local subspaces $L_e$ defined above, Rado's theorem gives the
rank function
\[
\min_{T\subseteq S}
\bigl(|S\setminus T|+\dim\sum_{a\in T}L_{e_a}\bigr).
\]
The centralizer calculation below identifies
$\dim\sum_{a\in T}L_{e_a}$ with $\rho_d(T)$.  The essential realization step
is to prove that genuine FUNTF columns, despite being coupled by a common
moment equation, attain this ambient Rado rank.

The fiber formula has a direct interpretation.  The quantity
$|T|-\rho_d(T)$ measures the excess of missing angular degrees of freedom
over the number of independent tight-frame constraints available on the
subpattern $T$.  Its maximum is exactly the dimension of the generic
completion ambiguity.  In particular, finite generic completion is possible
only when $|S|\le\Nd$.

\begin{corollary}[Finite-completion criterion]\label{cor:finite}
In the setting of Theorem~\ref{thm:main}, the generic completion fiber is
finite if and only if
\[
|T|\le \rho_d(T)
\qquad\text{for every }T\subseteq S.
\]
\end{corollary}

Finite completion here does not assert uniqueness or specify the number
of completions.  At the critical length $R=\Nd$, if every column is partially observed, the
coordinate projection is square because exactly $\Nd(d-2)=\dim X_{d,\Nd}$
coordinates remain observed \cite[Theorem~3.3]{BFR2020}.

\begin{corollary}[Critical-length basis criterion]\label{cor:squarebasis}
Suppose $R=\Nd$ and every column has exactly two missing coordinates.  Let
$H$ be the associated $\Nd$-edge multigraph on $[d]$.  Then the observed
coordinates form a basis of the algebraic matroid of $X_{d,\Nd}$ if and only if
\[
|B|\le \rho_d(B)
\qquad\text{for every nonempty edge submultiset }B\subseteq E(H).
\]
Equivalently, the observed-coordinate projection is generically finite.
\end{corollary}

In dimension four the critical-length criterion collapses to a simple closed
condition.

\begin{corollary}[Four-dimensional specialization]\label{cor:d4intro}
For $X_{4,9}$ with exactly two missing coordinates in each column, the
observed coordinates form an algebraic-matroid basis if and only if the
associated nine-edge multigraph is connected and no vertex pair carries more
than six parallel edges.
\end{corollary}

The bound $\Nd$ is sufficient, but is not asserted to be optimal.  Global
dependencies do occur at shorter lengths.  Suppose $d+2\le R\le2d-2$ and every
column has the same missing pair $e$.  Then
\[
\sum_{a=1}^R q_a(W)
=
\left[J_e,\sum_{a=1}^R w_aw_a^T\right]
=
\left[J_e,\frac Rd I_d\right]=0
\]
for every $W\in X_{d,R}$, while the Rado formula would predict rank $R$
because $\rho_d(T)=2d-2$ for every nonempty parallel-edge submultiset
$T$.  Thus the theorem is a genuine high-redundancy stabilization result,
not a formal consequence of the local subspace arrangement at every frame
length.

In particular, the smallest uniform stabilization threshold lies between
$2d-1$ and $\Nd$ (Corollary~\ref{cor:threshold}).

The proof has two stages.  At the critical length $\Nd$, the local subspace
dimension splits into an off-diagonal term and ordinary graphic rank.
Matroid partition gives an off-diagonal basis together with a spanning tree;
a Hall matching, a zero-coordinate tournament, and a moment-submersion
argument realize the off-diagonal core.  A generic residual decomposition
along the tree then recovers the remaining diagonal directions.  This is the
Basis Realization Theorem of Section~\ref{sec:core}.

The second stage transports that realization to every $R\ge\Nd$.  We extend
an arbitrary Rado-independent subfamily to an auxiliary critical basis, use
the spanning-set stability under addition of fully observed columns from
\cite[Remark~4.6]{BFR2020}, and finally retain a nonzero minor involving only
the original labels.  The tangent-kernel identity then converts the stabilized
rank formula into the fiber-dimension formula.

Section~\ref{sec:structure} develops the tangent-space calculation and
the matroid decomposition.  Section~\ref{sec:core} proves the critical
realization theorem, and Section~\ref{sec:rankfiber} derives stabilization,
fiber dimensions, and the local real result.  Section~\ref{sec:consequences}
gives the matroid-union interpretation, the explicit four-dimensional
formula, and bounds on the stabilization threshold.  Appendix~\ref{app:d4}
verifies the tournament base case.

\section{Jacobian ranks and the matroid decomposition}\label{sec:structure}

Except in the explicitly stated general-length tangent-space lemmas,
Sections~\ref{sec:structure} and \ref{sec:core} work at the critical length
\[
R=\Nd=\binom{d+1}{2}-1.
\]

\subsection{The restricted Jacobian and local rank obstruction}\label{sec:jacobian}

For distinct $i,j$, recall $J_{ij}=E_{ij}-E_{ji}$ and set
\[
S_{ij}:=E_{ij}+E_{ji},\qquad D_{ij}:=E_{ii}-E_{jj}.
\]
Thus $S_{ji}=S_{ij}$ and $D_{ji}=-D_{ij}$.
Recall also that $\Symo(d)=\{A\in\Sym(d):\tr A=0\}$.  We use the
direct-sum decomposition
\[
\Symo(d)=\OO_d\oplus\DD_d,
\]
where
\[
\OO_d=\Span\{S_{ij}:1\le i<j\le d\},\qquad
\DD_d=\left\{\operatorname{diag}(x_1,\dots,x_d):\sum_i x_i=0\right\}.
\]
Thus
\[
\dim\OO_d=\binom d2,\qquad \dim\DD_d=d-1.
\]
Let $P_{\OO_d}$ and $P_{\DD_d}$ denote the linear projections associated
with this direct sum.

\begin{lemma}[A regular locus for the frame equations]\label{lem:regular}
Let $d\ge4$ and $R\ge d+2$.  The set
\[
X_{d,R}^{\circ}=\{W\in X_{d,R}:w_a^Tw_b\ne0\text{ for all }a\ne b\}
\]
is a nonempty Zariski-open subset of $X_{d,R}$.  Every point of this set is
smooth, of dimension $R(d-1)-\Nd$, and its tangent space is
\[
T_WX_{d,R}=\left\{(\dot w_a)_{a=1}^R:
 w_a^T\dot w_a=0\ (1\le a\le R),\quad
 \sum_a(w_a\dot w_a^T+\dot w_aw_a^T)=0\right\}.
\]
\end{lemma}

\begin{proof}
The Gram matrix $K=W^TW$ satisfies $K^2=(R/d)K$ and $K_{aa}=1$.
Thus
\[
\sum_{b\ne a}(w_a^Tw_b)^2=\frac Rd-1\ne0.
\]
Some off-diagonal Gram entry is nonzero at some frame, and column
permutations show that each specified off-diagonal entry is a nonzero
regular function on $X_{d,R}$.  Irreducibility
\cite[Theorem~3.3]{BFR2020} now implies that their nonvanishing loci have
nonempty intersection.

Let $\mathcal Q_d=\{w\in\C^d:w^Tw=1\}$, a smooth quadric, and consider
\[
M:\mathcal Q_d^R\longrightarrow\Symo(d),\qquad
M(W)=\sum_a w_aw_a^T-\frac Rd I_d.
\]
A trace-free symmetric matrix $B$ annihilates $DM_W$ under the trace
pairing precisely when $Bw_a=\lambda_aw_a$ for every $a$.  Symmetry gives
$(\lambda_a-\lambda_b)w_a^Tw_b=0$.  At a point of $X_{d,R}^{\circ}$ all
$\lambda_a$ coincide.  Since $WW^T$ is invertible, the columns span
$\C^d$, so $B$ is scalar; its trace is zero, hence $B=0$.
Consequently $DM_W$ is surjective.  Its zero fiber is smooth of codimension
$\Nd$ in $\mathcal Q_d^R$, and its tangent space is the displayed kernel.
\end{proof}

\begin{lemma}[Tangent kernel for missing pairs]\label{lem:tangent}
Let $d\ge4$, $R\ge d+2$, and let $S\subseteq[R]$ carry prescribed
missing-pair labels.  If $W\in X_{d,R}^{\circ}$ and
$(w_{a,i},w_{a,j})\ne(0,0)$ whenever $e_a=\{i,j\}$, then
\[
\ker D\pi_S|_W\simeq
\left\{(t_a)_{a\in S}:\sum_{a\in S}t_aq_a(W)=0\right\}.
\]
These hypotheses hold on a nonempty Zariski-open subset of $X_{d,R}$.
\end{lemma}

\begin{proof}
A tangent vector in the kernel is supported on the missing pairs.  The
linearized norm equation in column $a$ has one-dimensional kernel on its
missing coordinate plane, generated by $J_{e_a}w_a$.  Thus
$\dot w_a=t_aJ_{e_a}w_a$ for $a\in S$ and $\dot w_a=0$ otherwise.
Since $J_{e_a}^T=-J_{e_a}$,
\[
\dot w_aw_a^T+w_a\dot w_a^T
=t_a[J_{e_a},w_aw_a^T]=t_aq_a(W).
\]
Lemma~\ref{lem:regular} gives both directions of the asserted equivalence.

For each fixed column and pair, the condition that both entries vanish
is a proper closed subset: a complex orthogonal change of row coordinates
can send that unit column to either coordinate vector of the pair.
Intersecting the complements of these finitely many subsets with
$X_{d,R}^{\circ}$ proves the last assertion.
\end{proof}

Recall that
\[
L_{ij}=\im\bigl(\operatorname{ad}_{J_{ij}}:\Sym(d)\to\Symo(d)\bigr)
=\{[J_{ij},X]:X\in\Sym(d)\}.
\]
The subspaces $L_{ij}$ therefore give a universal upper bound on the rank
of every restricted family of vectors $q_a$.  The main issue of this paper
is whether the common FUNTF moment constraint creates additional generic
dependencies beyond those forced by the subspaces $L_{ij}$.  Section~\ref{sec:rankfiber}
shows that it does not.

For a nonempty simple support graph $G$ on $[d]$, let $V(G)$ be its set of
nonisolated vertices, let $v=|V(G)|$, and let $c$ be its number of nontrivial
connected components.  Put
\[
L_G=\sum_{ij\in E(G)}L_{ij}\subseteq\Symo(d).
\]

\begin{lemma}[Centralizer rank formula]\label{lem:centralizer}
One has
\[
\dim L_G=\frac{v(2d-v+1)}2-c.
\]
Consequently, if a submultiset $B$ of missing-edge copies has support $G$, then
any family $\{q_a:a\in B\}$ is linearly dependent whenever
$|B|>\rho_d(B)$.
\end{lemma}

\begin{proof}
Use the trace pairing on $\Symo(d)$.  A matrix $S\in\Symo(d)$ is orthogonal
to $L_G$ if and only if
\[
\langle S,[J_{ij},X]\rangle=0
\qquad\text{for every }ij\in E(G),\ X\in\Sym(d),
\]
which is equivalent to
\[
[S,J_{ij}]=0
\qquad(ij\in E(G)).
\]
On each nontrivial connected component $C$ of $G$, the rotations
$\{J_{ij}:ij\in E(C)\}$ generate $\mathfrak{so}(C)$.  Indeed, for pairwise
distinct $i,j,k$ one has the ordered-index identity
\[
[J_{ij},J_{jk}]=J_{ik},
\]
and repeated commutators along a path generate $J_{uv}$ for every pair of
vertices $u,v\in C$.  A symmetric matrix commuting with all of
$\mathfrak{so}(C)$ is scalar on the coordinate subspace of $C$.  If $B$ is
a block joining $C$ to a different component or to an isolated coordinate,
commutation gives $J_{uv}B=0$ for every $u,v\in C$.  If $b$ is any column
of $B$, then $J_{uv}b=0$ forces the $u$- and $v$-coordinates of $b$ to
vanish; varying $u,v\in C$ forces every coordinate of $b$ on $C$ to be
zero.  Hence $B=0$.  Thus every
such cross block vanishes.  On the $d-v$ isolated
coordinates there is no restriction, so the symmetric centralizer in
$\Sym(d)$ has dimension
\[
c+\binom{d-v+1}{2}.
\]
Intersecting with the trace-zero hyperplane lowers this by one.  Therefore
\begin{align*}
\dim L_G
&=\dim\Symo(d)-\dim L_G^\perp\\
&=\left(\binom{d+1}{2}-1\right)
 -\left(c+\binom{d-v+1}{2}-1\right)\\
&=\frac{v(2d-v+1)}2-c.
\end{align*}
Every $q_a$ whose edge label belongs to $G$ lies in $L_G$, which gives
the stated necessary inequalities for edge submultisets.
\end{proof}
\subsection{The matroid structure of admissible patterns}\label{sec:matroid}

At the critical length, call an $\Nd$-edge multigraph \emph{admissible} if
$|B|\le\rho_d(B)$ for every nonempty edge submultiset $B$.
For an edge $e=\{i,j\}$, recall $L_e=L_{ij}$ and let
$U_e:=P_{\OO_d}(L_e)$.

\begin{lemma}[Off-diagonal edge space]\label{lem:Ue}
For $e=\{i,j\}$,
\[
U_e=\Span\left(S_{ij},\{S_{ik}:k\ne i,j\},\{S_{jk}:k\ne i,j\}\right).
\]
Hence, for an edge set $B$ incident with exactly $v=v(B)$ vertices,
\[
\dim\sum_{e\in B}U_e
=\sigma_d(B):=\binom d2-\binom{d-v}{2}
=\frac{v(2d-v-1)}2.
\]
\end{lemma}

\begin{proof}
For $k\notin\{i,j\}$, the off-diagonal parts of the basic commutators include
\[
P_{\OO_d}[J_{ij},E_{ii}]=-S_{ij},\qquad
P_{\OO_d}[J_{ij},S_{ik}]=-S_{jk},\qquad
P_{\OO_d}[J_{ij},S_{jk}]=S_{ik}.
\]
Conversely, inspection of the matrix product shows that no off-diagonal
coordinate whose two endpoints both lie outside $\{i,j\}$ can occur.  Hence
$U_e$ is exactly the displayed coordinate span.  For an active vertex set
$V$ of size $v$, the sum of the $U_e$ is therefore the span of all $S_{ab}$
with $\{a,b\}\cap V\ne\varnothing$. Counting gives the formula.
\end{proof}

Let $r_{\mathrm{gr}}(B)=v(B)-c(B)$ be the graphic-matroid rank.

\begin{corollary}[Rank splitting]\label{cor:split}
For every nonempty edge set $B$,
\[
\rho_d(B)=\sigma_d(B)+r_{\mathrm{gr}}(B).
\]
\end{corollary}

Conceptually, the off-diagonal projection records which symmetric
coordinates $S_{ij}$ can be touched by each missing edge and therefore forms
a coordinate-subspace Rado arrangement.  The complementary diagonal
directions are the incidence directions $D_{ij}$, whose independence is
graphic.  The identity above is the dimension-level shadow of this
Rado--graphic decomposition.

Let $M_{\off}$ be the Rado matroid induced by the family $\{U_e\}$, and let
$r_{\off}$ denote its rank function.  We use standard matroid terminology as
in \cite{Oxley,Schrijver}.  For notational
convenience set $\rho_d(\varnothing)=\sigma_d(\varnothing)=
r_{\mathrm{gr}}(\varnothing)=0$.

\begin{lemma}[Off-diagonal--graphic partition]\label{lem:partition}
Let $H$ be an $\Nd$-edge multigraph on $[d]$ and assume
$|B|\le\rho_d(B)$ for every nonempty edge submultiset
$B\subseteq E(H)$. Then
\[
E(H)=A\sqcup T_{\mathrm{gr}},
\]
where $|A|=\binom d2$, $|T_{\mathrm{gr}}|=d-1$, $A$ is a basis of
$M_{\off}$, and $T_{\mathrm{gr}}$ is a spanning tree.
\end{lemma}

\begin{proof}
The Rado rank formula~\cite{Rado1942,McDiarmid1975} gives
\[
r_{\off}(X)=\min_{Y\subseteq X}\bigl(|X\setminus Y|+\sigma_d(Y)\bigr).
\]
Hence
\begin{align*}
r_{\off}(X)+r_{\mathrm{gr}}(X)
&\ge \min_{Y\subseteq X}\bigl(|X\setminus Y|+\sigma_d(Y)+r_{\mathrm{gr}}(Y)\bigr)\\
&=\min_{Y\subseteq X}\bigl(|X\setminus Y|+\rho_d(Y)\bigr)\\
&\ge |X|.
\end{align*}
Thus $|X|\le r_{\off}(X)+r_{\mathrm{gr}}(X)$ for every $X\subseteq E(H)$.
The two-matroid partition criterion~\cite{Edmonds1965,EdmondsFulkerson1965}
therefore partitions $E(H)$ into a set independent in $M_{\off}$ and a set
independent in the graphic matroid. Since
\[
|E(H)|=\Nd=\binom d2+(d-1)
\]
is the sum of the two ambient ranks, both sets have maximum possible
cardinality. The second is therefore a spanning tree.
\end{proof}

Whenever an admissible multigraph $H$ is fixed in the remainder of the
critical-length argument, we choose and retain one partition
$E(H)=A\sqcup T_{\mathrm{gr}}$ furnished by
Lemma~\ref{lem:partition}.

\paragraph{Coordinate matching and zero-coordinate branches.}
Let $K_d$ denote the complete graph on the vertex set $[d]$.

\begin{lemma}[Coordinate matching]\label{lem:matching}
Let $A$ be a basis of $M_{\off}$. Then there is a bijection
\[
\phi:A\to E(K_d)
\]
such that $S_{\phi(a)}\in U_{e_a}$ for every $a\in A$. Equivalently,
$\phi(a)\cap e_a\ne\varnothing$.
\end{lemma}

\begin{proof}
For each $a\in A$, let $\mathcal S_a$ be the set of coordinate basis vectors
contained in $U_{e_a}$.  Each $U_{e_a}$ is a coordinate subspace of
$\OO_d$ with coordinate basis $\mathcal S_a$; hence the sum
$\sum_{a\in B}U_{e_a}$ has coordinate basis
$\bigcup_{a\in B}\mathcal S_a$.  Therefore, for every $B\subseteq A$,
\[
\left|\bigcup_{a\in B}\mathcal S_a\right|
=\dim\sum_{a\in B}U_{e_a}\ge |B|.
\]
Hall's theorem~\cite{Hall1935} yields a system of distinct representatives, and because
$|A|=\binom d2$ it is a bijection onto all off-diagonal coordinates.
\end{proof}

Choose $h(a)\in e_a\cap\phi(a)$ and restrict column $a$ to
\[
Q_{h(a)}=\{w\in\C^d:w^Tw=1,\ w_{h(a)}=0\}.
\]

\begin{lemma}[Zero-branch Jacobian span]\label{lem:zero-star}
Let $e=\{h,j\}$ and suppose $w_h=0$.  Define
$\varepsilon(e,h)\in\{\pm1\}$ by
$J_e=\varepsilon(e,h)J_{hj}$.  Then
\[
[J_e,ww^T]
=\varepsilon(e,h)\left(
 w_j^2S_{hj}+\sum_{k\ne h,j}w_jw_kS_{hk}
\right).
\]
Consequently
\[
\Span\{[J_e,ww^T]:w\in Q_h\}
=\operatorname{Star}(h):=\Span\{S_{hk}:k\ne h\}.
\]
\end{lemma}

\begin{proof}
The displayed formula follows from matrix multiplication using
$J_e=\varepsilon(e,h)J_{hj}$, so every image lies in the stated star.
Conversely, $w=e_j$ gives the direction
$\varepsilon(e,h)S_{hj}$.  For each $k\ne h,j$, take
\[
w=\frac{e_j+e_k}{\sqrt2}\in Q_h.
\]
Its image is
$\frac{\varepsilon(e,h)}2(S_{hj}+S_{hk})$, and subtracting the already
available $S_{hj}$ direction yields $S_{hk}$.  Hence the linear span of the
actual zero-branch images is exactly $\operatorname{Star}(h)$.  Notice that
we assert a statement about linear span, not that every vector in the star
is itself a single quadratic image.
\end{proof}

\begin{lemma}[Pure off-diagonal basis can be realized]\label{lem:pure-core}
There exist choices $w_a\in Q_{h(a)}$, $a\in A$, for which
\[
\{[J_{e_a},w_aw_a^T]:a\in A\}
\]
is a basis of $\OO_d$.
\end{lemma}

\begin{proof}
Fix a coordinate basis of $\OO_d$ and let
\[
D((x_a)_{a\in A})
\]
denote the determinant whose columns are the coordinate vectors of the
$x_a\in\OO_d$.  The function $D$ is separately linear in every argument.
Suppose, toward a contradiction, that
\[
D\bigl(([J_{e_a},w_aw_a^T])_{a\in A}\bigr)=0
\]
for every choice $w_a\in Q_{h(a)}$.  Fixing all factors except one and using
separate linearity, the vanishing extends from the actual image in that
factor to its linear span.  Lemma~\ref{lem:zero-star} identifies that span
with $\operatorname{Star}(h(a))$.  Repeating this argument one factor at a
time shows that $D$ vanishes on
\[
\prod_{a\in A}\operatorname{Star}(h(a)).
\]
But $S_{\phi(a)}\in\operatorname{Star}(h(a))$ for every $a$, and the matched
coordinates $S_{\phi(a)}$ are precisely the distinct coordinate basis
vectors of $\OO_d$.  Their determinant is nonzero, a contradiction.
Therefore an actual zero-branch choice with nonzero determinant exists.
\end{proof}

\section{Geometric realization of admissible bases}\label{sec:core}

\subsection{Zero-coordinate tournaments and the critical core}

Transport the heads chosen in Section~\ref{sec:matroid} through the
matching $\phi$: orient each coordinate edge $\phi(a)$ toward $h(a)$.
This produces a tournament because $\phi$ is a bijection and
$h(a)\in\phi(a)$.  We first prove a moment-map statement for any such
tournament; its branch factors depend only on the chosen heads.

For a tournament $\mathcal T$ on $[d]$, let $h(e)$ denote the terminal
vertex of $e\in E(K_d)$; the zero branch $Q_{h(e)}$ sets that coordinate
to zero.  Outdegree counts edges directed away from a vertex.  Define
\[
\mathcal B_{\mathcal T}=\prod_{e\in E(K_d)}Q_{h(e)}.
\]
Consider
\[
\Phi_{\mathcal T}:\mathcal B_{\mathcal T}\to
\{S\in\Sym(d):\tr S=\tbinom d2\},\qquad
(w_e)\mapsto\sum_e w_ew_e^T.
\]

\begin{proposition}[Tournament-zero moment submersion]\label{prop:tournament}
For every tournament $\mathcal T$ on $[d]$, $d\ge4$, the map
$\Phi_{\mathcal T}$ is submersive at some point. Moreover one may choose an
edge $e_0$ such that, after deleting that factor, the remaining moment map is
still submersive onto its trace hyperplane.
\end{proposition}

\begin{proof}
We argue by induction on $d$.  The $d=4$ case is verified in
Appendix~\ref{app:d4}.  For each of the four tournament isomorphism classes,
one edge is deleted and five explicit zero-branch vectors are chosen.  The
only element of $\Symo(4)$ annihilating all five factor differentials is then
shown to be zero.  Since $\dim\Symo(4)=9$, the five-factor differential has
rank $9$ and is submersive onto the trace hyperplane.  The full six-factor
map is therefore submersive as well.

Assume $d\ge5$ and choose a vertex $v$ of outdegree at least two; relabel it
as $1$.  Deleting vertex $1$ leaves a tournament on $\{2,\dots,d\}$.  Apply
the induction hypothesis to its full moment map.  Thus there are old branch
vectors $\widetilde w_e\in\C^{d-1}$ for which the old differential spans
$\Symo(d-1)$.  Embed them into $\C^d$ as $w_e=(0,\widetilde w_e)$.
All prescribed heads are among the vertices $2,\dots,d$, so these embedded
vectors remain admissible.  The tangent variables of the different factors vary
independently, so a matrix annihilating the differential of the product
moment map must annihilate the differential contributed by each factor
separately.  For the old factors we may therefore restrict to tangent
variations $(0,\delta\widetilde w_e)$.  If $S\in\Symo(d)$ annihilates
the full differential, its lower-right block has size
$(d-1)\times(d-1)$ and annihilates the old trace-zero differential.  The
annihilator of $\Symo(d-1)$ inside $\Sym(d-1)$ is the scalar line, so
\begin{equation}\label{eq:core-1}
S=
\begin{pmatrix}
-(d-1)\lambda & b^T\\
b&\lambda I_{d-1}
\end{pmatrix}
\end{equation}
for some $\lambda\in\C$ and $b=(b_2,\dots,b_d)^T$.

Choose two outgoing edges from vertex $1$ and relabel their heads as $2$ and
$3$.  On these two factors take the raw branch vectors
\[
u=(1,0,1,1,\dots,1),\qquad
z=(1,1,0,1,2,1,\dots,1),
\]
and normalize them.  Both raw vectors have nonzero squared norm, and
normalization does not change the annihilator condition.  For a factor with
head $h$ and branch vector $w\in Q_h$, a symmetric matrix $S$ annihilates its
moment differential exactly when
\begin{equation}\label{eq:core-2}
P_{H_h}(Sw)\in\C w,
\qquad H_h=\{x_h=0\},
\end{equation}
where $P_{H_h}$ is the coordinate projection onto $H_h$.  This equivalence
follows because the tangent space is $T_wQ_h=H_h\cap w^\perp$.

Apply \eqref{eq:core-2} first to $u$, whose head is $2$.  Let $c$ be the
proportionality factor in $P_{H_2}(Su)=cu$.  For every $j=3,\dots,d$,
\[
(Su)_j=b_j+\lambda=c,
\]
so
\[
b_j=c-\lambda\qquad (j=3,\dots,d).
\]
The first coordinate of $P_{H_2}(Su)=cu$ then gives
\[
-(d-1)\lambda+\sum_{j=3}^d b_j=c.
\]
The sum contains $d-2$ terms.  Substituting $b_j=c-\lambda$ therefore gives
\[
-(d-1)\lambda+(d-2)(c-\lambda)=c,
\]
or equivalently
\begin{equation}\label{eq:core-3}
(d-3)c=(2d-3)\lambda.
\end{equation}

Apply \eqref{eq:core-2} next to $z$, whose head is $3$, and let $\beta$ be the
corresponding proportionality factor.  Since $z_4=1$ and $z_5=2$, while
$b_4=b_5=c-\lambda$, coordinates $4$ and $5$ give
\[
\beta=(Sz)_4=b_4+\lambda=c
\]
and
\[
2\beta=(Sz)_5=b_5+2\lambda=c+\lambda.
\]
Hence $c=\lambda$.  Substituting this into \eqref{eq:core-3} gives
\[
(d-3)\lambda=(2d-3)\lambda,
\]
so $d\lambda=0$.  Since we work over $\C$, $\lambda=0$, and therefore
$c=0$ and $b_j=0$ for $j=3,\dots,d$.  Finally the second coordinate of
$P_{H_3}(Sz)=\beta z$ gives $b_2+\lambda=\beta=c$, hence $b_2=0$ as well.
Thus $b=0$ and $S=0$, so the full differential is submersive.

It remains to preserve submersivity after deleting one factor.  If vertex
$1$ has at least three outgoing edges, delete one other than the two bridge
edges used above.  If it has exactly two outgoing edges, then its indegree is
$d-3\ge2$ for $d\ge5$, so an incoming edge at vertex $1$ is available; delete
such an edge.  In either case all old tournament factors and the same two outgoing
bridge factors remain, so the preceding annihilator calculation is unchanged.
The resulting $(\binom d2-1)$-factor map is therefore still submersive.
\end{proof}

\subsection{Forcing the critical rank drop}

Set $k=\binom d2$.  Let $f_0\in E(K_d)$ be a deletable tournament edge
supplied by Proposition~\ref{prop:tournament}.  Since $\phi:A\to E(K_d)$ is
a bijection, define
\[
a_0:=\phi^{-1}(f_0)\in A,
\qquad
A':=A\setminus\{a_0\}.
\]

\begin{lemma}[Simultaneous generic core conditions]\label{lem:simultaneous}
On
\[
\mathcal P=\prod_{a\in A'}Q_{h(a)}
\]
there is a nonempty Zariski-open subset on which:
\begin{enumerate}[label=\textup{(\alph*)}]
\item the $k-1$ zero-branch Jacobians are independent in $\OO_d$;
\item $\operatorname{Star}(h(a_0))$ is not contained in their span;
\item the moment map of these $k-1$ columns is submersive.
\end{enumerate}
\end{lemma}

\begin{proof}
Each $Q_{h(a)}$ is an irreducible affine quadric: it is cut out in the
$(d-1)$-dimensional coordinate hyperplane $H_{h(a)}$ by the nondegenerate
equation $w^Tw=1$, with $d-1\ge3$.  Hence the product $\mathcal P$ is
irreducible.

We first combine (a) and (b) into one explicit rank condition.  Choose a
zero-branch basis point given by Lemma~\ref{lem:pure-core}, and let
\[
V_*:=\Span\{q_a:a\in A'\}
\]
at that point.  Since the full family indexed by $A$ is a basis of $\OO_d$,
the pivot vector $q_{a_0}$ does not lie in $V_*$.  On the other hand,
$q_{a_0}\in\operatorname{Star}(h(a_0))$.  Therefore not every vector in the
fixed coordinate basis
\[
\{S_{h(a_0)\ell}:\ell\ne h(a_0)\}
\]
of $\operatorname{Star}(h(a_0))$ can lie in $V_*$.  Fix one such basis vector
$s_0\notin V_*$.  The condition
\[
\rank\bigl(q_a:a\in A',\ s_0\bigr)=k
\]
is the nonvanishing of a $k\times k$ minor.  It is therefore a nonempty
Zariski-open condition on $\mathcal P$.  Wherever it holds, the $k-1$
vectors indexed by $A'$ are independent and $s_0$ is outside their span;
thus both (a) and (b) hold.

Condition (c) is the nonvanishing of a maximal minor of the moment
differential and is nonempty by the deleted-factor assertion in
Proposition~\ref{prop:tournament}.  Hence (c) also holds on a nonempty
Zariski-open subset of $\mathcal P$.  The two nonempty opens intersect because
$\mathcal P$ is irreducible.
\end{proof}

For $p\in\mathcal P$ write
\[
G_0(p)=\frac{\Nd}{d}I-\sum_{a\in A'}w_aw_a^T.
\]
Then $p\mapsto G_0(p)$ is dominant onto the affine trace-$d$ hyperplane:
condition (c) says that its differential is surjective at one point, which is
enough for dominance.  Let $e_0=e_{a_0}=\{h,j\}$ and
$H_h=\{x\in\C^d:x_h=0\}$.

The next two lemmas provide the pivot mechanism for the critical core.  The
first identifies the linear span of rank-one matrices on a generic quadratic
section; the second uses that span to choose the missing pivot on the
rank-drop locus without losing the full star of Jacobian directions.

\begin{lemma}[Restricted quadratic span]\label{lem:quadspan}
Let $A\in\Sym(H_h)$ be generic and define
\[
C_A=\{w\in H_h:w^Tw=1,\ w^TAw=0\}.
\]
Then
\[
\Span\{ww^T:w\in C_A\}=A^\perp\subset\Sym(H_h).
\]
\end{lemma}

\begin{proof}
Work on the nonempty Zariski-open locus where the characteristic polynomial
of $A$ has simple roots; this is the nonvanishing locus of its discriminant
and contains every diagonal matrix with pairwise distinct entries.  Then $A$
has an eigenbasis.  Since $A$ is symmetric, eigenvectors for distinct
eigenvalues are orthogonal for the standard bilinear form.  None can be
isotropic: otherwise it would be orthogonal to the whole eigenbasis,
contradicting nondegeneracy.  After rescaling, choose a complex orthogonal
basis in which $A=\operatorname{diag}(a_1,\dots,a_{d-1})$, and put
$z_i=w_i^2$.  The equations are
\[
\sum_i z_i=1,\qquad \sum_i a_iz_i=0.
\]
Pairwise distinct $a_i$ make these two affine constraints independent, so
their common zero set $P$ is an affine $(d-3)$-plane.  It is nonempty and is
not contained in any coordinate hyperplane.  Indeed, for any fixed index
$k$, choose two distinct indices $\ell,m\ne k$; this is possible because
$d-1\ge3$.  Set $z_k=1$ and all coordinates except
$z_k,z_\ell,z_m$ equal to zero.  The two defining equations of $P$ then form
a nonsingular $2\times2$ system for $z_\ell,z_m$, since
$a_\ell\ne a_m$.  Hence $P$ contains a point with $z_k\ne0$ for every
$k$, and the complement in $P$ of the coordinate hyperplanes is a dense
nonempty open subset.  Every point of this subset lifts to points of $C_A$
by choosing square roots.

Suppose $S=(s_{ij})\in\Sym(H_h)$ annihilates all $ww^T$ with $w\in C_A$.
Fix such a $z$ with every coordinate nonzero and choose square roots
$r_i^2=z_i$.  For every sign vector $\varepsilon\in\{\pm1\}^{d-1}$, the
vector $w_i=\varepsilon_i r_i$ still lies in $C_A$.  Thus
\[
0=w^TSw=\sum_i s_{ii}z_i
   +2\sum_{i<j}s_{ij}\varepsilon_i\varepsilon_j r_i r_j
\]
for every sign vector.  The constant function together with the characters
$\varepsilon_i\varepsilon_j$ on $\{\pm1\}^{d-1}$ is a linearly independent
family, so every Fourier coefficient of this expression vanishes.  In
particular, $s_{ij}r_i r_j=0$ for every $i<j$.  Since all $r_i$ are nonzero, every
off-diagonal entry of $S$ vanishes.  Hence
$S=\operatorname{diag}(s_i)$.
The homogeneous linear form $\ell(z)=\sum_i s_i z_i$ vanishes on the
affine plane
\[
P=\left\{z:\sum_i z_i=1,\ \sum_i a_i z_i=0\right\}.
\]
The affine linear forms vanishing on $P$ are exactly the span of
$\sum_i z_i-1$ and $\sum_i a_i z_i$.  Hence
\[
\ell(z)=\alpha\left(\sum_i z_i-1\right)
       +\beta\sum_i a_i z_i.
\]
Comparing constant terms gives $\alpha=0$, because $\ell$ is homogeneous.
Thus $s_i=\beta a_i$ for every $i$, so $S=\beta A$.  Therefore the
annihilator of the span is exactly $\C A$, which is equivalent to the stated
identity with $A^\perp$.
\end{proof}

\begin{lemma}[Rank-drop star span]\label{lem:rankdrop-star}
For generic $G_0\in\{G\in\Sym(d):\tr G=d\}$, the vectors
\[
w\in H_h,\qquad w^Tw=1,\qquad \det(G_0-ww^T)=0
\]
have Jacobian images under $w\mapsto[J_{e_0},ww^T]$ spanning all of
$\operatorname{Star}(h)$.
\end{lemma}

\begin{proof}
For generic $G_0$ we may assume that $G_0$ is invertible.  The matrix
determinant lemma gives
\[
\det(G_0-ww^T)=\det(G_0)(1-w^TG_0^{-1}w).
\]
Thus the rank-drop locus inside $H_h$ is
\begin{equation}\label{eq:core-4}
w^Tw=1,\qquad w^TAw=0,
\qquad
A=(G_0^{-1}-I)|_{H_h}.
\end{equation}
Lemma~\ref{lem:quadspan} therefore gives
\[
\Span\{ww^T:w\text{ satisfies }\eqref{eq:core-4}\}=A^\perp\subset\Sym(H_h).
\]

Let
\[
L=L_{e_0,h}:\Sym(H_h)\longrightarrow\operatorname{Star}(h),
\qquad X\longmapsto [J_{e_0},X].
\]
Lemma~\ref{lem:zero-star} implies that $L$ is surjective.  We claim that
\begin{equation}\label{eq:core-5}
L(A^\perp)=\operatorname{Star}(h)
\quad\Longleftrightarrow\quad
A\notin\operatorname{im}L^*,
\end{equation}
where the adjoint is taken with respect to the trace pairing.  Indeed, the
annihilator of $L(A^\perp)$ consists of those
$\eta\in\operatorname{Star}(h)^*$ for which
$L^*\eta\in(A^\perp)^\perp=\C A$.  Since $L$ is surjective, $L^*$ is
injective.  Therefore $L(A^\perp)$ is proper precisely when some nonzero
$L^*\eta$ is a scalar multiple of $A$, which is equivalent to
$A\in\operatorname{im}L^*$.  Moreover
\[
\dim\operatorname{im}L^*=d-1
<\dim\Sym(H_h)=\frac{d(d-1)}2,
\]
so $\operatorname{im}L^*$ is a proper linear subspace.

It remains to see that the exceptional condition in \eqref{eq:core-5} is genuinely
nongeneric in $G_0$.  On the open subset of the trace-$d$ hyperplane where
$G_0$ is invertible, consider the rational map
\[
F(G_0)=(G_0^{-1}-I)|_{H_h}.
\]
Its image contains a dense open subset of $\Sym(H_h)$.  To see this, take
any $A$ for which $I+A$ is invertible and
\[
\gamma=d-\tr((I+A)^{-1})\ne0,
\]
and set, relative to $\C^d=H_h\oplus\C e_h$,
\[
G_0=\operatorname{diag}((I+A)^{-1},\gamma).
\]
Then $G_0$ is invertible, $\tr G_0=d$, and $F(G_0)=A$.  Hence $F$ is
dominant and generic $G_0$ avoids the proper subspace
$\operatorname{im}L^*$.  Combining \eqref{eq:core-4}--\eqref{eq:core-5} with
Lemma~\ref{lem:quadspan}, the rank-drop Jacobian images span the whole
$\operatorname{Star}(h)$.
\end{proof}

\begin{proposition}[Critical off-diagonal core with rank-$d-1$ residual]\label{prop:critical}
There exist unit columns labelled by the edges in $A$ such that
\[
\Span\{q_a:a\in A\}=\OO_d
\]
and
\[
G=\frac{\Nd}{d}I-\sum_{a\in A}w_aw_a^T
\]
has $\rank G=d-1$ and $\tr G=d-1$.
\end{proposition}

\begin{proof}
Let $\mathcal U_{\rm rd}$ be the nonempty Zariski-open subset of the
trace-$d$ hyperplane on which Lemma~\ref{lem:rankdrop-star} applies.  Since
$p\mapsto G_0(p)$ is dominant, the inverse image of $\mathcal U_{\rm rd}$
under this map is a nonempty Zariski-open subset of $\mathcal P$.  Intersect
it with the nonempty open set of Lemma~\ref{lem:simultaneous} and choose
$p$ there.
Lemma~\ref{lem:rankdrop-star} then allows the pivot column to be chosen on
the rank-drop locus with
\[
q_{a_0}\notin\Span\{q_a:a\in A'\}.
\]
Thus all $k$ core Jacobians form a basis of $\OO_d$. The residual
$G=G_0-w_{a_0}w_{a_0}^T$ is singular. Since $G_0$ is invertible and a
rank-one update lowers rank by at most one, $\rank G=d-1$. Finally
$\tr G=\Nd-k=d-1$.
\end{proof}
\subsection{Residual geometry and spanning-tree completion}\label{sec:residual}

The critical core leaves a symmetric residual of trace and generic rank
$d-1$.  We first show that the residuals produced by such cores fill a
nonempty open subset of the residual determinantal variety.  We then show
that generic residuals admit unit rank-one decompositions compatible with the
fixed spanning tree; combining the two steps completes the basis realization.

Let
\[
\RR_{d-1}=\{G\in\Sym(d):\tr G=d-1,\ \rank G\le d-1\}.
\]

\begin{lemma}[Irreducibility of the residual variety]\label{lem:Rirr}
The variety $\RR_{d-1}$ is irreducible.
\end{lemma}

\begin{proof}
Put $r=d-1$ and consider
\[
\mathcal B_r
=
\left\{B\in\C^{d\times r}:\tr(BB^T)=r\right\}.
\]
The defining polynomial is a nondegenerate affine quadratic polynomial in
$dr\ge3$ variables.  It is irreducible over $\C$: a factorization would
have to be into affine linear factors, forcing the quadratic part to have
rank at most two, whereas here its rank is $dr\ge3$.  Hence $\mathcal B_r$
is irreducible.  The morphism
\[
\pi:\mathcal B_r\longrightarrow\RR_{d-1},
\qquad B\longmapsto BB^T,
\]
is surjective.  Indeed, if $G$ has rank $s\le r$, congruence
diagonalization of complex symmetric matrices gives
$G=C\operatorname{diag}(I_s,0)C^T$ for some invertible $C$.  Taking the first
$s$ columns of $C$ and padding with $r-s$ zero columns produces
$B\in\C^{d\times r}$ with $G=BB^T$.  Since $G\in\RR_{d-1}$,
$\tr(BB^T)=\tr G=r$, so $B\in\mathcal B_r$.  Hence the image
$\RR_{d-1}$ is irreducible.
\end{proof}

\begin{proposition}[Residual dominance]\label{prop:residual-dominance}
The residual matrices arising from critical off-diagonal cores satisfying
$\Span\{q_a:a\in A\}=\OO_d$ contain a
nonempty Zariski-open subset of $\RR_{d-1}$.
\end{proposition}

\begin{proof}
Choose $p\in\mathcal P$ and a pivot $w\in Q_h$ as in the proof of
Proposition~\ref{prop:critical}.  At this choice the $A'$ moment map is
submersive, the full core Jacobians span $\OO_d$, and the residual has
rank $d-1$.  Put
\[
G=G_0(p)-ww^T
\]
and let $z$ span $\ker G$.  Because $\rank G=d-1$, the determinant has
nonzero differential at $G$, represented up to scale by $H\mapsto z^THz$.
This functional is not a scalar multiple of the trace functional, since
$zz^T$ is not a scalar multiple of $I_d$.  Hence $G$ is a smooth point of
$\RR_{d-1}$ and
\begin{equation}\label{eq:core-6}
T_G\RR_{d-1}
=
\{H\in\Sym(d):\tr H=0,\ z^THz=0\}.
\end{equation}

Consider the incidence hypersurface
\[
\mathcal I
=
\{(p',w')\in\mathcal P\times Q_h:
  \det(G_0(p')-w'w'^T)=0\}.
\]
We first verify smoothness at $(p,w)$.  Each quadric $Q_{h(a)}$ and $Q_h$
is smooth, because the gradient of $w^Tw-1$ is $2w\ne0$ on the quadric.
Thus $\mathcal P\times Q_h$ is smooth.  Since $G$ has corank one,
$\operatorname{adj}(G)$ is a nonzero scalar multiple of $zz^T$.  Holding
$w$ fixed, the differential of the defining equation in a variation
$\delta G_0$ is therefore a nonzero scalar multiple of
$z^T(\delta G_0)z$.  Submersivity of the $A'$ moment map means that the
allowed variations $\delta G_0$ fill the entire trace-zero hyperplane
$\Symo(d)$.  The functional $H\mapsto z^THz$ is not identically zero on
$\Symo(d)$, because the rank-one matrix $zz^T$ is not a scalar multiple of
$I_d$.  Hence the defining equation of $\mathcal I$ has nonzero differential
at $(p,w)$, so $\mathcal I$ is smooth there.

Let
\[
\mathcal R:\mathcal I\longrightarrow\RR_{d-1},
\qquad
\mathcal R(p',w')=G_0(p')-w'w'^T
\]
be the residual map.  Take any $H\in T_G\RR_{d-1}$.  Set
$\delta w=0$.  Because the $A'$ moment map is submersive, there is a tangent
variation $\delta p$ for which $\delta G_0=H$.  By \eqref{eq:core-6},
$z^THz=0$, so $(\delta p,0)$ satisfies the linearized incidence equation.
Thus it is tangent to $\mathcal I$ and
$d\mathcal R(\delta p,0)=H$.  Therefore $d\mathcal R$ is surjective at the
smooth point $(p,w)$.

Because $\mathcal I$ is smooth at $(p,w)$, that point lies on a unique
irreducible component; denote it by $\mathcal C$.  The differential of
$\mathcal R|_{\mathcal C}$ at $(p,w)$ has rank
$\dim T_G\RR_{d-1}=\dim\RR_{d-1}$.  Therefore the Zariski closure of
$\mathcal R(\mathcal C)$ has dimension at least $\dim\RR_{d-1}$.
Since the image is already contained in $\RR_{d-1}$, the closure must equal
$\RR_{d-1}$; equivalently, $\mathcal R|_{\mathcal C}$ is dominant.  The
condition that the core Jacobians span $\OO_d$ is Zariski open and holds at
$(p,w)$, so it defines a nonempty dense open subset
$\mathcal C^\circ\subset\mathcal C$.  Its image remains dense in the
same irreducible target.  Chevalley's theorem makes
$\mathcal R(\mathcal C^\circ)$ constructible; a dense constructible subset
of an irreducible variety contains a nonempty Zariski-open subset.  Hence a
nonempty open subset of $\RR_{d-1}$ consists of residuals arising from
critical off-diagonal cores with the required full-rank property.
\end{proof}

\paragraph{Generic unit decompositions.}

We need a unit rank-one decomposition only for a generic residual in
$\RR_{d-1}$.

Recall the quadric
\[
\mathcal Q_d=\{u\in\C^d:u^Tu=1\}.
\]
It has dimension $d-1$ and is irreducible for $d\ge4$.
Indeed, its defining quadratic polynomial is nondegenerate and cannot factor
into affine linear factors.  Put $r=d-1$ and define
\[
\Psi_r:\mathcal Q_d^r\longrightarrow\RR_{d-1},
\qquad
(u_1,\dots,u_r)\longmapsto\sum_{j=1}^r u_ju_j^T.
\]
The image lies in $\RR_{d-1}$ because every factor has unit norm, so the
sum has trace $r=d-1$, and it has rank at most $r$.

\begin{lemma}[Generic residual decomposition dominance]\label{lem:generic-completion}
The morphism $\Psi_r$ is dominant.  More precisely, there is a nonempty
Zariski-open subset of $\mathcal Q_d^r$ on which $\Psi_r$ is submersive.
\end{lemma}

\begin{proof}
Let $U=(u_1,\dots,u_r)$ have linearly independent columns, assume that the
Gram matrix $U^TU$ is nonsingular, and assume
\[
u_i^Tu_j\ne0\qquad(i\ne j).
\]
Such unit vectors exist explicitly.  Choose $t\in\C$ with
$t\ne0$, $1+t^2\ne0$, and $1+rt^2\ne0$, and take
\[
u_j=\frac{e_j+t e_d}{\sqrt{1+t^2}},\qquad 1\le j\le r.
\]
Then $u_i^Tu_j=t^2/(1+t^2)\ne0$ for $i\ne j$, and the Gram matrix has
eigenvalues $(1+t^2)^{-1}$ with multiplicity $r-1$ and
$(1+rt^2)/(1+t^2)$ with multiplicity one.  Hence it is nonsingular.

The differential is
\[
D\Psi_r(\delta u_1,\dots,\delta u_r)
=
\sum_j(u_j\delta u_j^T+\delta u_ju_j^T),
\qquad
u_j^T\delta u_j=0.
\]
A symmetric matrix $S$ annihilates its image if and only if
$Su_j\in\C u_j$ for every $j$.  Write $Su_j=\lambda_j u_j$.  Symmetry of $S$
gives
\[
(\lambda_i-\lambda_j)u_i^Tu_j=0,
\]
so all $\lambda_j$ are equal, say to $\lambda$.  Because $U^TU$ is
nonsingular, $V=\Span\{u_j\}$ is a nondegenerate $(d-1)$-plane.  Hence
\[
\C^d=V\oplus V^\perp,
\qquad V^\perp=\C z,
\]
with $z^Tz\ne0$.  We have already shown that $S$ acts by the same scalar
$\lambda$ on every $u_j$, and therefore on all of $V$.  Symmetry of $S$
forces $V^\perp$ to be invariant, say $Sz=\mu z$.  Thus
\[
S=\lambda I+\frac{\mu-\lambda}{z^Tz}zz^T,
\]
and the annihilator of $D\Psi_r$ is exactly
\begin{equation}\label{eq:core-7}
\Span\{I,zz^T\}.
\end{equation}

Now $G=UU^T$ has rank $r=d-1$, trace $r$, and kernel $V^\perp=\C z$.
Near such a point the residual variety is cut out inside the trace-$r$
hyperplane by the determinant equation.  Since
$\operatorname{adj}(G)$ is a nonzero scalar multiple of $zz^T$, its tangent
space is
\[
T_G\RR_{d-1}
=
\{H\in\Sym(d):\tr H=0,\ z^THz=0\},
\]
whose normal space under the trace pairing is precisely the two-dimensional
space in \eqref{eq:core-7}.  Because the image of $D\Psi_r$ is contained in this tangent
space and has the same annihilator, the two spaces are equal.  Hence
$\Psi_r$ is submersive at $U$ and is therefore dominant.
\end{proof}

\paragraph{The spanning-tree block.}

\begin{lemma}[Generic tree-compatible decomposition]\label{lem:treecompatible}
Let $T_{\mathrm{gr}}$ be a spanning tree on $[d]$.  There is a nonempty
Zariski-open subset
$\RR_{T_{\mathrm{gr}}}^\circ\subseteq\RR_{d-1}$ such that every
$G\in\RR_{T_{\mathrm{gr}}}^\circ$ admits a decomposition
\[
G=\sum_{t\in T_{\mathrm{gr}}}u_tu_t^T,
\qquad
u_t^Tu_t=1,
\]
with
\[
u_{t,i}u_{t,j}\ne0
\qquad
(t=\{i,j\}\in T_{\mathrm{gr}}).
\]
\end{lemma}

\begin{proof}
Label the $r=d-1$ factors of $\mathcal Q_d^r$ by the edges of
$T_{\mathrm{gr}}$.  The conditions $u_{t,i}u_{t,j}\ne0$ define a
Zariski-open subset $\mathcal U_{\mathrm{gr}}$.  It is nonempty: for each
tree edge $t=\{i,j\}$, the choice
\[
u_t=\frac{e_i+e_j}{\sqrt2}
\]
belongs to $\mathcal Q_d$ and satisfies $u_{t,i}u_{t,j}=1/2$.  Since
$\mathcal Q_d^r$ is irreducible, this nonempty open subset intersects the
nonempty submersion open set from Lemma~\ref{lem:generic-completion}.  Hence
the restriction of $\Psi_r$ to $\mathcal U_{\mathrm{gr}}$ remains dominant.
Its image is dense constructible in the irreducible target $\RR_{d-1}$ and
therefore contains a nonempty Zariski-open subset
$\RR_{T_{\mathrm{gr}}}^\circ$.
\end{proof}

\begin{theorem}[Basis realization]\label{thm:basisrealization}
Let $H$ be an $\Nd$-edge multigraph on $[d]$ satisfying
\[
|B|\le \rho_d(B)
\qquad\text{for every nonempty edge submultiset }B\subseteq E(H).
\]
Then there exists $W\in X_{d,\Nd}$ with edge labels prescribed by $H$ for which
the $\Nd$ restricted Jacobian columns form a basis of $\Symo(d)$.  Consequently
the same basis property holds on a nonempty Zariski-open subset of
$X_{d,\Nd}$.
\end{theorem}

\begin{proof}
By Lemma~\ref{lem:partition}, choose and fix a partition
\[
E(H)=A\sqcup T_{\mathrm{gr}},
\]
where $A$ is a basis of $M_{\off}$ and $T_{\mathrm{gr}}$ is a spanning tree.
For this choice of $A$, Proposition~\ref{prop:residual-dominance} gives a
nonempty open subset of $\RR_{d-1}$ arising from critical off-diagonal cores.
For this choice of $T_{\mathrm{gr}}$, Lemma~\ref{lem:treecompatible} gives
another nonempty open subset consisting of tree-compatible residuals.  Since
$\RR_{d-1}$ is irreducible, these two opens intersect.

Choose $G$ in their intersection.  The set $A$ consists of the original edge
copies of $H$ assigned to the off-diagonal block; each core column retains
its original missing-pair label $e_a$.  The critical core satisfies
\[
\Span\{q_a:a\in A\}=\OO_d,\qquad
\sum_{a\in A}w_aw_a^T=\frac{\Nd}{d}I-G.
\]
Likewise, $T_{\mathrm{gr}}$ is the set of original edge copies assigned
to the graphic block.  For each $t\in T_{\mathrm{gr}}$, write
$e_t=\{i,j\}$ for its original missing-pair label and assign to that column
the corresponding vector $u_t$ from a tree-compatible decomposition
\[
G=\sum_{t\in T_{\mathrm{gr}}}u_tu_t^T.
\]
Then
\[
P_{\DD_d}[J_{ij},u_tu_t^T]=\pm2u_{t,i}u_{t,j}D_{ij}.
\]
All weights are nonzero, and $\{D_{ij}:e_t=\{i,j\},\ t\in T_{\mathrm{gr}}\}$ is a basis
of $\DD_d$.  Indeed, after identifying trace-zero diagonal matrices with
vectors $x\in\C^d$ satisfying $\sum_i x_i=0$, these $D_{ij}$ are the oriented
edge-incidence vectors of the spanning tree $T_{\mathrm{gr}}$.  Thus the tree
Jacobians form a basis modulo the core span, and
all $\Nd$ Jacobian columns form a basis of
$\Symo(d)=\OO_d\oplus\DD_d$.  Every core and tree column has bilinear norm
one, and the two displayed moment identities give
\[
\sum_{a\in A}w_aw_a^T+\sum_{t\in T_{\mathrm{gr}}}u_tu_t^T=\frac{\Nd}{d}I.
\]
Hence the union of the labelled core and tree columns is a point of
$X_{d,\Nd}$ with exactly the original missing-pair labels of $H$.  The determinant of these $\Nd$ Jacobian
columns is therefore a nonzero regular function on $X_{d,\Nd}$, so its
nonvanishing locus is a nonempty Zariski-open subset.
\end{proof}

\section{High-redundancy rank stabilization and completion fibers}\label{sec:rankfiber}

We now transport the critical-length realization to every frame length
$R\ge\Nd$.  Here a coordinate set is called \emph{spanning} if its rank in
the algebraic matroid of $X_{d,R}$ equals $\dim X_{d,R}$; equivalently, its
coordinate projection is generically finite onto its image.  The only
additional input is the following stability fact for spanning coordinate
sets.

\begin{lemma}[Spanning stability under a fully observed column]\label{lem:column-extension}
Let $R\ge d+2$, and let $E\subseteq[d]\times[R]$ be spanning in
$X_{d,R}$.  Then
\[
E^+:=E\cup\{(1,R+1),\ldots,(d,R+1)\}
\]
is spanning in $X_{d,R+1}$.
\end{lemma}

\begin{proof}
This is the column-extension observation of
Bernstein--Farnsworth--Rodriguez \cite[Remark~4.6]{BFR2020}.  Their variety
$X_{n,r}$ is the same complex algebraic FUNTF variety used here: it is the
Zariski closure in $\C^{n\times r}$ defined by the unit-norm and tight-frame
equations.  The cited remark states that a spanning coordinate set in
$X_{n,r}$ remains spanning in $X_{n,r+1}$ after all $n$ coordinates of the
new column are adjoined.  We use only this forward implication.  In our
application $d\ge4$ and $R\ge\Nd\ge d+2$, so every iteration lies in the
irreducible range $r\ge n+2>4$ of \cite[Theorem~3.3]{BFR2020}.
\end{proof}

\begin{proposition}[Stabilized generic restricted rank]\label{prop:genericrank}
Let $R\ge\Nd$ and let $S\subseteq[R]$ be any set of frame columns with
prescribed edge labels.  For generic $W\in X_{d,R}$,
\[
\rank\{q_a(W):a\in S\}
=
\min_{T\subseteq S}
\bigl(|S\setminus T|+\rho_d(T)\bigr).
\]
Moreover all subset ranks agree simultaneously with the corresponding Rado
rank function on a single nonempty Zariski-open subset of $X_{d,R}$.
\end{proposition}

\begin{proof}
Since $q_a(W)\in L_{e_a}$, for every $T\subseteq S$,
\[
\rank\{q_a(W):a\in S\}
\le
|S\setminus T|+\dim\sum_{a\in T}L_{e_a}
=
|S\setminus T|+\rho_d(T).
\]
Taking the minimum gives the upper bound.

For the reverse inequality we separate the argument into three steps.  The
auxiliary edge labels introduced below are used only to manufacture a witness
point; the final polynomial minor will involve only the original independent
subfamily and its original labels.

\smallskip\noindent\emph{Step 1: extend to a critical Rado basis.}
Let $I\subseteq S$ be a maximum independent set
of the Rado matroid induced by the subspaces $\{L_{e_a}\}_{a\in S}$.
Rado's rank formula gives
\[
|I|=
\min_{T\subseteq S}
\bigl(|S\setminus T|+\rho_d(T)\bigr),
\]
and every $J\subseteq I$ satisfies $|J|\le\rho_d(J)$.

Adjoin to $I$ a finite reservoir of $\Nd$ fresh labelled copies of each
edge of $K_d$, carrying the corresponding subspaces $L_e$.  Its Rado rank
is $\Nd$: distinct copies can represent the basis
\[
\{S_{ij}:1\le i<j\le d\}
\cup
\{D_{1j}:2\le j\le d\}
\quad\text{of }\Symo(d),
\]
since $S_{ij},D_{ij}\in L_{ij}$, and the ambient dimension is $\Nd$.

By the basis-extension axiom, $I$ extends to a reservoir basis $B$ of
cardinality $\Nd$.  Hence every $C\subseteq B$ satisfies
\[
|C|\le\dim\sum_{a\in C}L_{e_a}=\rho_d(C).
\]
Choose $\Nd$ distinct frame-column positions containing all positions of
$I$, assign the added edge copies to the remaining chosen positions, and
keep the original labels on $I$.  No compatibility with the original labels
on $S\setminus I$ is required.  The auxiliary labels serve only to construct
one point at which a Jacobian minor involving the columns of $I$ is nonzero.
That minor depends only on the frame entries in the columns indexed by $I$
and on their original labels $e_a$, $a\in I$; labels assigned to auxiliary
columns do not occur in it.

\smallskip\noindent\emph{Step 2: realize the auxiliary basis and transport it to length $R$.}
At the critical length $\Nd$, Theorem~\ref{thm:basisrealization} gives a
nonempty open set on which the $B$-labelled restricted Jacobians form a basis
of $\Symo(d)$.  Intersect
this open set with $X_{d,\Nd}^{\circ}$ from Lemma~\ref{lem:regular}.  If both missing coordinates of one column vanished, its restricted Jacobian
would be zero; hence the basis property also guarantees that every missing
pair is nonzero.  By Lemma~\ref{lem:tangent}, the Jacobian-basis property gives zero kernel for
the observed-coordinate differential.  Thus the corresponding coordinate
set is spanning in $X_{d,\Nd}$.

After a column permutation, regard these $\Nd$ positions as the first
$\Nd$ columns.  Apply Lemma~\ref{lem:column-extension} successively
$R-\Nd$ times, observing every coordinate in each added column.  We obtain a
spanning coordinate set in $X_{d,R}$ whose only missing entries are the two
prescribed coordinates in the $B$-labelled columns.

Since this projection $\eta_B$ is generically finite onto its image and
the ground field has characteristic zero, its differential is injective on
a nonempty open subset of the smooth locus.  Intersect this subset with the
open set in Lemma~\ref{lem:tangent}.  At a point of the intersection,
$\ker D\eta_B=0$ implies that the $\Nd$ vectors $q_b(W)$, $b\in B$, are
independent.  In particular the subfamily indexed by $I$ is independent at
some point of $X_{d,R}$.

\smallskip\noindent\emph{Step 3: return to the original labelled pattern.}
An $|I|\times|I|$ minor certifying this independence is a polynomial in
the frame entries of the columns in $I$ and their original labels only.
The auxiliary construction shows that this polynomial does not vanish
identically on $X_{d,R}$; changing labels outside $I$ does not change it.
Its nonvanishing locus therefore gives generic independence for the
original pattern.

The generic rank on $S$ is therefore at least $|I|$, proving equality.
Apply the same argument to every subset of $S$.  Since $X_{d,R}$ is
irreducible, the finitely many resulting nonempty Zariski-open sets have
nonempty intersection, which gives simultaneous validity of all subset ranks.
\end{proof}

\begin{proposition}[Stabilized generic fiber dimension]\label{prop:fiberdim}
Let $R\ge\Nd$.  Let $\pi_S$ forget exactly the two coordinates labelled by
$e_a$ in each column $a\in S$ and observe every other coordinate.  Then the
generic fiber dimension is
\[
\dim \pi_S^{-1}(\pi_S(W))
=
\max_{T\subseteq S}\bigl(|T|-\rho_d(T)\bigr).
\]
\end{proposition}

\begin{proof}
Work on the nonempty open subset of $X_{d,R}^{\circ}$ where
Proposition~\ref{prop:genericrank} has its generic rank and the hypotheses
of Lemma~\ref{lem:tangent} hold.
By Lemma~\ref{lem:tangent},
\[
\dim\ker D\pi_S
=|S|-\rank\{q_a(W):a\in S\}
=|S|-r_d(S).
\]

Let $Y=Y_S=\overline{\pi_S(X_{d,R})}$.  Since $R\ge\Nd\ge d+2$,
$X_{d,R}$ is irreducible by \cite[Theorem~3.3]{BFR2020}.  For a variety $Z$, write $Z^{\mathrm{sm}}$ for its smooth locus.
The restriction
\[
\pi_S:\;X_{d,R}^{\mathrm{sm}}\cap\pi_S^{-1}(Y^{\mathrm{sm}})
\longrightarrow Y^{\mathrm{sm}}
\]
is dominant.  By generic smoothness in characteristic zero, after shrinking
this source it is smooth.  On that dense open subset its differential has
rank $\dim Y$.  Hence
\[
\dim\ker D\pi_S=\dim X_{d,R}-\dim Y,
\]
which is also the generic fiber dimension by the fiber-dimension theorem.
Finally,
\[
|S|-r_d(S)
=
\max_{T\subseteq S}\bigl(|T|-\rho_d(T)\bigr).
\]
\end{proof}

\begin{proof}[Proof of Theorem~\ref{thm:main}]
The rank assertion is Proposition~\ref{prop:genericrank}, and the
fiber-dimension assertion is Proposition~\ref{prop:fiberdim}.
\end{proof}

\begin{proof}[Proof of Corollary~\ref{cor:finite}]
The generic fiber is finite exactly when its dimension is zero.  By
Theorem~\ref{thm:main}, this is equivalent to
$|T|-\rho_d(T)\le0$ for every $T\subseteq S$.
\end{proof}

\begin{proof}[Proof of Corollary~\ref{cor:squarebasis}]
The observed set has cardinality $\Nd(d-2)=\dim X_{d,\Nd}$
\cite[Theorem~3.3]{BFR2020}.  Its projection is generically finite exactly
under the inequalities of Corollary~\ref{cor:finite}.  In that case its
image closure has the dimension of the ambient affine coordinate space
and hence is the whole space.  Thus the observed coordinates are both
spanning and algebraically independent, so form a basis.  Conversely, a
basis is spanning and its coordinate projection is generically finite.
\end{proof}

\begin{corollary}[Generic real rank and local completion dimension]\label{cor:real}
Let $d\ge4$, $R\ge\Nd$, and fix a labelled missing-pair pattern $S$.
There is a nonempty Zariski-open subset $\Omega_S\subset X_{d,R}$ defined
over $\mathbb R$ and contained in the smooth locus such that
$\Omega_S(\mathbb R)$ is Euclidean open and dense in
\[
X_{d,R}^{\mathrm{sm}}(\mathbb R)
=X_{d,R}^{\mathrm{sm}}\cap\mathbb R^{d\times R}.
\]
At every $W\in\Omega_S(\mathbb R)$, all missing-pair subset ranks equal
$r_d(S')$, and the real completion fiber is, near $W$, a smooth real
manifold of dimension $\delta_d(S)$.  In particular, $\delta_d(S)=0$
gives local uniqueness of the real completion at such $W$.
\end{corollary}

\begin{proof}
Intersect $X_{d,R}^{\circ}$ with the open conditions in
Lemma~\ref{lem:tangent} and the simultaneous maximal-rank conditions of
Theorem~\ref{thm:main}.  All these conditions are defined over $\mathbb R$:
the matrices $q_a(W)$ have real polynomial entries, and maximal rank is
specified by the nonvanishing of their minors.  The resulting open subset
$\Omega_S$ is nonempty.  The real FUNTF locus is Zariski dense in
$X_{d,R}$ by its identification with the complex Zariski closure in
\cite[Section~1 and equation~(1.1)]{BFR2020}; hence
$\Omega_S(\mathbb R)$ is nonempty.

Every nonempty Euclidean open subset of the smooth real locus is Zariski
dense in $X_{d,R}$.  Indeed, near a smooth real point the real implicit
function theorem gives real analytic coordinates that extend to complex
analytic coordinates on $X_{d,R}$.  A polynomial vanishing on a real open
neighborhood vanishes on the corresponding complex neighborhood by the
identity theorem, and hence on the irreducible variety.  The proper closed
complement of $\Omega_S$ therefore has empty interior in the smooth real
locus, proving the stated Euclidean density.

At a real point, a real matrix has the same rank over $\mathbb R$ and
$\C$.  Lemmas~\ref{lem:regular} and \ref{lem:tangent}, applied over
$\mathbb R$, give
\[
\dim_{\mathbb R}\ker D(\pi_S|_{X_{d,R}(\mathbb R)})_W
=|S|-r_d(S)=\delta_d(S).
\]
This rank is constant on $\Omega_S(\mathbb R)$.  The real constant-rank
theorem gives the local manifold assertion, including local uniqueness
when the dimension is zero.
\end{proof}

This corollary concerns the local fiber at a generic smooth real frame.
It neither asserts that arbitrary real observations are completable nor
determines the number of real completions elsewhere in the fiber.

\section{Consequences and specializations}\label{sec:consequences}

\subsection{Matroid union and critical bases}
Let $M_{\off}$ be the off-diagonal Rado matroid from
Section~\ref{sec:matroid}, and let $M_{\mathrm{gr}}$ be the graphic matroid,
both on the labelled edge copies of a fixed set $S$.

\begin{proposition}[Matroid-union description]\label{prop:union}
The Rado matroid induced by $\{L_{e_a}:a\in S\}$ is
$M_{\off}\vee M_{\mathrm{gr}}$.  Thus, for $R\ge\Nd$, the missing-pair
Jacobian matroid is the union of a transversal matroid and a graphic matroid.
\end{proposition}

\begin{proof}
An edge subfamily $I$ is Rado-independent exactly when
$|B|\le\rho_d(B)$ for every $B\subseteq I$.  The same rank calculation
as in Lemma~\ref{lem:partition}, with no cardinality assumption, yields
$|B|\le r_{\off}(B)+r_{\mathrm{gr}}(B)$ for all $B\subseteq I$.
The matroid-partition theorem partitions $I$ into an $M_{\off}$-independent
set and a graphic-independent set.  Conversely, any such partition gives
\[
|B|\le r_{\off}(B)+r_{\mathrm{gr}}(B)
\le\sigma_d(B)+r_{\mathrm{gr}}(B)=\rho_d(B).
\]
This proves the matroid-union identity.  Since the spaces $U_e$ are
coordinate subspaces, their Rado matroid is the transversal matroid of the
bipartite graph joining $e$ to each $f\in E(K_d)$ with $e\cap f\ne\varnothing$.
The last assertion follows from Theorem~\ref{thm:main}.
\end{proof}

Consequently standard matroid-union algorithms \cite{Schrijver}, with
matching and forest independence tests, compute $r_d(S)$ and
$\delta_d(S)=|S|-r_d(S)$ without enumerating all edge submultisets.
At the critical length, bases have the following concrete form.

\begin{corollary}[Partition form of critical-length admissibility]\label{cor:partitionform}
For an $\Nd$-edge multigraph $H$ on $[d]$, the inequalities of
Corollary~\ref{cor:squarebasis} hold if and only if the edge copies can be
partitioned as
\[
E(H)=A\sqcup T_{\mathrm{gr}},
\]
where $T_{\mathrm{gr}}$ is a spanning tree and $A$ has cardinality
$\binom d2$ and admits
a bijection
\[
\phi:A\longrightarrow E(K_d)
\]
such that $e_a\cap\phi(a)\ne\varnothing$ for every $a\in A$.
\end{corollary}

\begin{proof}
The forward implication is Lemmas~\ref{lem:partition} and
\ref{lem:matching}. Conversely, the existence of such a partition makes
$A$ independent in $M_{\off}$ and $T_{\mathrm{gr}}$ independent in the graphic matroid.
Thus every edge subset $B$ satisfies
\[
|B|\le r_{\off}(B)+r_{\mathrm{gr}}(B)
\le \sigma_d(B)+r_{\mathrm{gr}}(B)=\rho_d(B),
\]
where the middle inequality follows from the defining subspace rank of the
Rado matroid.
\end{proof}

\subsection{An explicit formula in dimension four}

For a labelled multigraph on $[4]$, let $m_{ij}$ be the number of copies
of $\{i,j\}$, put $m=\sum_{i<j}m_{ij}$, and write
\[
m(U)=\sum_{\{i,j\}\subseteq U}m_{ij}\qquad(U\subseteq[4]).
\]
Let $\mathfrak P_4$ be the set of the three partitions of $[4]$ into two
unordered pairs.

\begin{corollary}[Four-dimensional defect formula]\label{cor:d4defect}
For $d=4$, $R\ge9$, and any missing-pair pattern $S$, one has
\begin{equation}\label{eq:d4defect}
\begin{split}
\delta_4(S)=\max\biggl\{&0,\ m-9,\ \max_{i<j}(m_{ij}-6),\\
 &\max_{\substack{U\subseteq[4]\\|U|=3}}(m(U)-8),\
 \max_{\{\{i,j\},\{k,\ell\}\}\in\mathfrak P_4}
 (m_{ij}+m_{k\ell}-8)\biggr\}.
\end{split}
\end{equation}
In particular, $r_4(S)=m-\delta_4(S)$.
\end{corollary}

\begin{proof}
A nonempty support on two vertices has $\rho_4=6$; a support on three
vertices is connected and has $\rho_4=8$.  A support on all four vertices
has $\rho_4=9$ if connected and $\rho_4=8$ otherwise; in the latter case
it consists of two disjoint pairs.  Thus every term
$|T|-\rho_4(T)$ in Theorem~\ref{thm:main} is bounded above by one of the
quantities in \eqref{eq:d4defect}.

Conversely, take all edge copies on a specified pair, on a specified
three-vertex set, on a specified pair partition, or on all four vertices.
Their subspace dimensions are at most $6$, $8$, $8$, and $9$, respectively,
even if some of their prescribed supports are absent.  Each corresponding
quantity in \eqref{eq:d4defect} is therefore at most
$\max_{T\subseteq S}(|T|-\rho_4(T))$.  The empty subfamily gives zero,
proving equality.
\end{proof}

\begin{proof}[Proof of Corollary~\ref{cor:d4intro}]
When $d=4$ and $\Nd=9$, a two-vertex support has
\[
\rho_4=6,
\]
so no pair may carry more than six parallel copies.  The full edge set must
use all four vertices and be connected, since otherwise
$\rho_4(E(H))\le8<9$.

Conversely, assume connectivity and the multiplicity bound six.  Let $B$ be
any nonempty edge submultiset.  Since $H$ has exactly nine edges, $|B|\le9$.
If $v(B)=2$, then $|B|\le6=\rho_4(B)$ by hypothesis.  If $v(B)=3$, its
simple support is connected and $\rho_4(B)=8$; because $|B|\le9$, the only
possible violation is $|B|=9$.  That would place every edge of $H$ on three
vertices, contradicting connectivity of $H$.  If $v(B)=4$ and the support is
connected, then $\rho_4(B)=9$.  If it has two nontrivial components, then
$\rho_4(B)=8$; again, since $|B|\le9$, a violation would force $|B|=9$ and
hence make $H$ disconnected.  Thus all subset inequalities of
Corollary~\ref{cor:squarebasis} hold.
\end{proof}

\begin{example}[Local concentration and disconnected support]\label{ex:d4}
At $R=9$, write $k(ij)$ for $k$ copies of the missing pair $\{i,j\}$.
Formula~\eqref{eq:d4defect} gives
\[
\begin{array}{c|c|c}
\text{missing-pair multigraph}&r_4(S)&\delta_4(S)\\ \hline
6(12)+2(23)+(34)&9&0\\
7(12)+(23)+(34)&8&1\\
5(12)+4(34)&8&1\\
9(12)&6&3
\end{array}
\]
The first two patterns are connected, but the second exceeds the capacity
of a single pair.  The third respects every single-pair capacity and fails
because its two components together have capacity eight.  The last pattern
has three degrees of generic completion ambiguity.  The same ranks and
fiber dimensions hold if fully observed columns are added.
\end{example}

\subsection{Parallel edges and the stabilization threshold}

\begin{remark}[Parallel-edge test]
If $S$ consists of $k$ copies of a single edge, then every nonempty
$T\subseteq S$ has $\rho_d(T)=2d-2$.  Theorem~\ref{thm:main} therefore gives
\[
r_d(S)=\min\{k,2d-2\},
\qquad
\delta_d(S)=\max\{0,k-(2d-2)\}.
\]
For $d=4$, the single-edge capacity is six.
\end{remark}

\begin{remark}
For $d=3$, one has $\Nd=5$, and the critical-length basis specialization reduces to the
known three-dimensional criterion of Bernstein--Farnsworth--Rodriguez
\cite{BFR2020}.  The geometric realization proof above is designed for
$d\ge4$ because its tournament induction starts in dimension four.
\end{remark}

For fixed $d\ge4$, define $R_{\mathrm{stab}}(d)$ to be the least integer
$R_0\ge d+2$ such that the rank formula of Theorem~\ref{thm:main} holds
for every $R\ge R_0$ and every labelled missing-pair pattern.

\begin{corollary}[Bounds on the uniform stabilization threshold]\label{cor:threshold}
For every $d\ge4$,
\[
2d-1\le R_{\mathrm{stab}}(d)\le\binom{d+1}{2}-1.
\]
In particular, $7\le R_{\mathrm{stab}}(4)\le9$.
\end{corollary}

\begin{proof}
The upper bound is Theorem~\ref{thm:main}.  At $R=2d-2\ge d+2$, give
every column the same missing pair.  The identity $\sum_aq_a(W)=0$ from
the introduction forces rank at most $R-1$, whereas the Rado formula gives
$R$.  Thus a uniform threshold cannot be at most $2d-2$.
\end{proof}

The exact threshold within these bounds, patterns with larger missing
blocks, and global real completion counts remain open directions.
Corollary~\ref{cor:real} resolves the local real rank and dimension question
on the smooth generic locus, without addressing those global counts.

\appendix
\section{Direct verification of the four-dimensional tournament base case}\label{app:d4}

This appendix gives a direct verification of the four-dimensional base case
used in the tournament induction.  The argument is entirely through the
annihilator equations for the five-factor moment maps.  Since $\dim\Symo(4)=9$, showing that the common
annihilator in $\Symo(4)$ is zero is equivalent to showing that the moment
differential has rank $9$.

For a zero-branch factor with head $h$, let $w\in Q_h$, so that
$w_h=0$ and $w^Tw=1$.  Its tangent space is
\[
T_wQ_h=\{\delta\in\C^4:\delta_h=0,\ w^T\delta=0\}.
\]
If $S\in\Symo(4)$ annihilates the moment differential of this factor, then
\[
0=\langle S,w\delta^T+\delta w^T\rangle
  =2\delta^TSw
\qquad(\delta\in T_wQ_h),
\]
and hence
\begin{equation}\label{eq:ann-criterion}
Sw\in\Span\{w,e_h\}.
\end{equation}
Conversely, \eqref{eq:ann-criterion} implies annihilation.  Thus a collection
of branch factors has submersive moment map exactly when the only
$S\in\Symo(4)$ satisfying \eqref{eq:ann-criterion} for every factor is
$S=0$.

Up to permutation of the factors, the moment map depends on a tournament
only through the multiset of their heads.  For a branch product, define its
\emph{head multiplicity} to be the four-tuple $(m_1,m_2,m_3,m_4)$, where
$m_h$ counts the factors whose head (equivalently, prescribed zero
coordinate) is $h$.  Since heads are terminal vertices, $m_h$ is the
indegree of vertex $h$ in the full tournament.  Up to isomorphism, the four
tournaments on four vertices have the following sorted full head-count
sequences:
\[
(0,1,2,3),\qquad (0,2,2,2),\qquad
(1,1,2,2),\qquad (1,1,1,3).
\]
The required deletion can be read from the following table.  In each row,
choose a vertex having the head count shown in the middle column and delete
one edge directed into that vertex.  Such an edge exists because the count
is positive.  The right column lists the remaining head multiplicities after
resorting the four entries.
\[
\begin{array}{c|c|c}
\text{full counts}&\text{head count}&\text{remaining counts}\\ \hline
(0,1,2,3)&2&(3,1,1,0)\\
(0,2,2,2)&2&(2,2,1,0)\\
(1,1,2,2)&1&(2,2,1,0)\\
(1,1,1,3)&1&(3,1,1,0)
\end{array}
\]
Thus, after deleting a suitable edge, only the two five-factor head
multiplicities $(3,1,1,0)$ and $(2,2,1,0)$ remain to be checked.

Write
\[
S=
\begin{pmatrix}
a&p&q&r\\
p&b&s&t\\
q&s&c&u\\
r&t&u&-a-b-c
\end{pmatrix}.
\]
For a raw branch vector $w$ with $w^Tw\ne0$, choose a square root of
$w^Tw$ and set $\widehat w=w/\sqrt{w^Tw}\in Q_h$.  The criterion
\eqref{eq:ann-criterion} is invariant under replacing $w$ by any nonzero
scalar multiple, because both $Sw$ and the line $\C w$ scale accordingly.
Hence the annihilator calculation may be carried out with the raw integer
vectors displayed below.

\subsection*{Head multiplicity \texorpdfstring{$(3,1,1,0)$}{(3,1,1,0)}}

After relabeling, take three factors with head $4$, one with head $2$, and one
with head $3$.  Choose the raw vectors
\[
(1,0,-1,0),\qquad
(1,-1,0,0),\qquad
(0,1,1,0),\qquad
(1,0,0,-1),\qquad
(1,1,0,1),
\]
in that order, with zero coordinates $4,4,4,2,3$, respectively.
Applying \eqref{eq:ann-criterion} gives, among the resulting relations,
\begin{align*}
p-s&=0, & a-c&=0, & a-b&=0,\\
q-s&=0, & p+q&=0, & q-u&=0,\\
2a+b+c&=0, & -a+b-r+t&=0, & -2a-b-c-p+t&=0.
\end{align*}
The first five equations give $a=b=c$ and $p=q=s=0$.  Then $q-u=0$
gives $u=0$, while $2a+b+c=4a=0$ gives $a=b=c=0$.  The last two
equations reduce to $-r+t=0$ and $t=0$, hence $r=t=0$.  Therefore
$S=0$.

\subsection*{Head multiplicity \texorpdfstring{$(2,2,1,0)$}{(2,2,1,0)}}

After relabeling, take two factors with head $3$, two with head $4$, and one
with head $1$.  Choose
\[
\begin{aligned}
&(0,1,0,-1),\qquad (1,0,0,-1),\qquad (0,1,1,0),\\
&(1,-1,-1,0),\qquad (0,1,0,-1),
\end{aligned}
\]
with heads $3,3,4,4,1$, respectively.  Condition
\eqref{eq:ann-criterion} yields
\begin{align*}
p-r&=0, & a+2b+c&=0, & p-t&=0,\\
2a+b+c&=0, & p+q&=0, & -b+c&=0,\\
a-b-q-s&=0, & a-c-p-s&=0, & s-u&=0.
\end{align*}
Subtracting the second and fourth equations gives $a=b$; the sixth gives
$c=b$, and hence the second gives $4a=0$.  Thus $a=b=c=0$.  Next
$q=-p$.  The seventh relation gives $s=p$, whereas the eighth gives
$s=-p$, so $p=q=s=0$.  Finally $r=t=p=0$ and $u=s=0$.  Thus again
$S=0$.

In both possible head-multiplicity patterns, the five-factor moment
differential has zero annihilator in $\Symo(4)$ and hence rank $9$.  Since
every four-vertex tournament admits a deletion leading to one of these two
patterns, this proves both the four-dimensional base case and the
deleted-factor assertion in Proposition~\ref{prop:tournament} by an explicit
calculation.

\section*{Use of artificial intelligence}
During the preparation of this work, the author used OpenAI ChatGPT to
assist with literature searches, checking proofs and calculations, and
language polishing.  The author takes full responsibility for the content
of the article.

\end{document}